\documentclass[12pt,english,british]{amsart}
\usepackage[T1]{fontenc}
\usepackage[latin9]{inputenc}
\usepackage{geometry}
\usepackage{color}
\usepackage{babel}
\usepackage{amstext}
\usepackage{amsthm}
\usepackage{amssymb}
\usepackage{stackrel}
\usepackage{setspace}
\usepackage[unicode=true,pdfusetitle,
 bookmarks=true,bookmarksnumbered=false,bookmarksopen=false,
 breaklinks=false,pdfborder={0 0 0},pdfborderstyle={},backref=false,colorlinks=true]
 {hyperref}

\makeatletter
\numberwithin{equation}{section}
\numberwithin{figure}{section}
\theoremstyle{plain}
\newtheorem{thm}{\protect\theoremname}[section]
\theoremstyle{plain}
\newtheorem{cor}[thm]{\protect\corollaryname}
\theoremstyle{plain}
\newtheorem{conjecture}[thm]{\protect\conjecturename}
\theoremstyle{definition}
\newtheorem{defn}[thm]{\protect\definitionname}
\theoremstyle{plain}
\newtheorem{lem}[thm]{\protect\lemmaname}
\theoremstyle{plain}
\newtheorem{prop}[thm]{\protect\propositionname}
\theoremstyle{remark}
\newtheorem{rem}[thm]{\protect\remarkname}
\theoremstyle{definition}
\newtheorem{example}[thm]{\protect\examplename}

\makeatother

\addto\captionsbritish{\renewcommand{\conjecturename}{Conjecture}}
\addto\captionsbritish{\renewcommand{\corollaryname}{Corollary}}
\addto\captionsbritish{\renewcommand{\definitionname}{Definition}}
\addto\captionsbritish{\renewcommand{\examplename}{Example}}
\addto\captionsbritish{\renewcommand{\lemmaname}{Lemma}}
\addto\captionsbritish{\renewcommand{\propositionname}{Proposition}}
\addto\captionsbritish{\renewcommand{\remarkname}{Remark}}
\addto\captionsbritish{\renewcommand{\theoremname}{Theorem}}
\addto\captionsenglish{\renewcommand{\conjecturename}{Conjecture}}
\addto\captionsenglish{\renewcommand{\corollaryname}{Corollary}}
\addto\captionsenglish{\renewcommand{\definitionname}{Definition}}
\addto\captionsenglish{\renewcommand{\examplename}{Example}}
\addto\captionsenglish{\renewcommand{\lemmaname}{Lemma}}
\addto\captionsenglish{\renewcommand{\propositionname}{Proposition}}
\addto\captionsenglish{\renewcommand{\remarkname}{Remark}}
\addto\captionsenglish{\renewcommand{\theoremname}{Theorem}}
\providecommand{\conjecturename}{Conjecture}
\providecommand{\corollaryname}{Corollary}
\providecommand{\definitionname}{Definition}
\providecommand{\examplename}{Example}
\providecommand{\lemmaname}{Lemma}
\providecommand{\propositionname}{Proposition}
\providecommand{\remarkname}{Remark}
\providecommand{\theoremname}{Theorem}

\begin{document}
\title[Generalization of Jordan-Lie inner ideals of Finite Dimensional Associative
Algebras]{\singlespacing{} Generalization of Jordan-Lie of Finite Dimensional Associative Algebras}
\author{}
\author{Hasan M. S. Shlaka}
\address{Department of Mathematics, Faculty of Computer Science and Mathematics,
University of Kufa, Al-Najaf, Iraq.}
\email{hasan.shlaka@uokufa.edu.iq}
\begin{abstract}
We generalize Baranov and Shlaka's results about bar-minimal Jordan-Lie
and regular inner ideals of finite dimensional associative algebras.
Let $A$ be a finite dimensional $1$-perfect associative algebras
$A$ over an algebraically closed field $\mathbb{F}$ of arbitrary
characteristic $p$ and let $A'$ be a subalgebra of $A$. We prove
that for any bar-minimal Jordan-Lie inner ideal $B'$ a $A'$, there
is a bar-minimal Jordan-Lie inner ideal of $A$ that contains $B'$
and if $B$' is regular, then is a regular inner ideal of $A$ that
contains $B'$. We also prove that for any strict orthogonal pair
$(e',f')$ in $A'$, there is a strict orthogonal idempotent pair
$(e,f)$ in $A$ such that $e'A'f'\subseteq eAf$\href{http://.}{.}
\end{abstract}

\maketitle
\begin{singlespace}

\section{Introduction}
\end{singlespace}

\begin{singlespace}
\global\long\def\ad{\operatorname{\rm ad}}%

\global\long\def\bbR{\mathbb{R}}%

\global\long\def\bbF{\mathbb{F}}%

\global\long\def\ccR{\mathcal{R}}%

\global\long\def\ccF{\mathcal{F}}%

\global\long\def\ccM{\mathcal{M}}%

\global\long\def\ccL{\mathcal{L}}%

\global\long\def\ccP{\mathcal{P}}%

\global\long\def\ccB{\mathcal{B}}%

\end{singlespace}

\selectlanguage{english}%
\global\long\def\dlim{\operatorname{\underrightarrow{{\rm lim}}}}%

\selectlanguage{british}%
\begin{singlespace}
\global\long\def\core{\operatorname{\rm core}}%

\global\long\def\dim{\operatorname{\rm dim}}%

\global\long\def\End{\operatorname{\rm End}}%

\global\long\def\gl{\operatorname{\rm \mathfrak{gl}}}%

\global\long\def\Im{\operatorname{\rm Im}}%

\global\long\def\ker{\operatorname{\rm ker}}%

\global\long\def\rad{\operatorname{\rm rad}}%

\global\long\def\rank{\operatorname{\rm rank}}%

\global\long\def\Range{\operatorname{\rm Range}}%

\global\long\def\sym{\operatorname{\rm sym}}%

\global\long\def\skew{\operatorname{\rm skew}}%

\global\long\def\sl{\operatorname{\rm \mathfrak{sl}}}%

\global\long\def\sp{\operatorname{\rm \mathfrak{sp}}}%

\global\long\def\so{\operatorname{\rm \mathfrak{so}}}%

\global\long\def\su{\operatorname{\rm \mathfrak{su^{*}}}}%

\global\long\def\u{\operatorname{\rm \mathfrak{u^{*}}}}%

\end{singlespace}

 Let $L$ be a Lie algebra over a field $\mathbb{F}$. Recall that
a subspace $B$ of $L$ is said to be an \emph{inner ideal} of $L$
if $[B,[B,L]]\subseteq B$. It is well-known that every ideal is inner,
but inner ideals are not even subalgebras. Lie Inner ideals were first
introduced by Benkart \cite{Benkart1976}. In \cite{Benkart1977},
Benkart highlighted the relation between abelian inner ideals and
the adjoint nilpotent elements of Lie algebra and proved that the
she showed that abelian inner ideals and the adjoint nilpotent elements
of Lie algebras are related closely. Since the adjoint nilpotent elements
yield an elementary criterion for distinguishing the non-classical
from classical simple Lie algebras in positive characteristic, inner
ideals play a fundamental role in classifying Lie algebras \cite{Premet1987,Premet1986}.
Inner ideals are useful in constructing grading for Lie algebras \cite{LopGarLozNeh2009}.
It was shown in \cite{LopGarLoz2008} that inner ideals play role
similar to that of one-sided ideals in associative algebras and can
be used to develop Artinian structure theory for Lie algebras. Inner
ideals of classical Lie algebras were classified by Benkart and Fernández
López \cite{Benkart1976,BenkartLopez}, using the fact that most of
these algebras can be constructed as the derived Lie subalgebras of
(involution) simple Artinian associative rings. In this paper we use
a similar approach to study inner ideals of the derived Lie subalgebras
of finite dimensional associative algebras. These algebras generalize
the class of simple Lie algebras of classical type and are closely
related to the so-called root-graded Lie algebras \cite{Baranov2018Arx}.
They are also important in developing representation theory of non-semisimple
Lie algebras. As we do not require our algebras to be semisimple we
have a lot more inner ideals to take care of (including all ideals!),
so some reasonable restrictions are needed. We believe that such a
restriction is the notion of a Jordan-Lie inner ideals introduced
by Fernández López in \cite{FLopez2006} (see also \cite{FLopez2019}).
Baranov and Shlaka in \cite{BavShk} studied Jordan-Lie inner ideals
of finite dimensional associative algebras. 

In this paper, we generalize Baranov and Shlaka's results to the case
when the Jordan-Lie inner ideals are of subalgebras of of those associative
ones. We believe that this results is important to classify Jordan-Lie
inner ideals of simple locally finite associative algebras. We need
some notations to state our main results. The ground field $\bbF$
is algebraically closed of arbitrary characteristic $p$. Let $A$
be a finite dimensional associative algebra over $\bbF$ and let $R$
be the radical of $A$. Recall that $A$ becomes a Lie algebra $A^{(-)}$
under $[x,y]=xy-yx$. Put $A^{(0)}=A^{(-)}$ and $A^{(k)}=[A^{(k-1)},A^{(k-1)}]$,
$k\geq1$. Let $L=A^{(k)}$ for some $k\ge0$ and let $B$ be an inner
ideal of $L$. Then $B$ is said to be \emph{Jordan-Lie inner ideal}
(or simply \emph{J-Lie}) of $L$ if $B^{2}=0$. A J-Lie $B$ is called
\emph{regular} if $BAB\subseteq B$. Denote by $\bar{B}$ the image
of $B$ in $\bar{L}=L/R\cap L$. Let $X$ be an inner ideal of $\bar{L}$.
We say that $B$ is \emph{$X$-minimal} (or simply, \emph{bar-minimal})
if $\bar{B}=X$ and for every inner ideal $B'$ of $L$ with $\bar{B}'=X$
and $B'\subseteq B$ we have $B'=B$. Let $e$ and $f$ be idempotents
in $A$. Then the\emph{ idempotent pair} $(e,f)$ is said to be \emph{orthogonal
}if $ef=fe=0$ and \emph{strict }if for each simple component \emph{$S$
}of $\bar{A}=A/R$, the projections of $\bar{e}$ and $\bar{f}$ on
$S$ are either both zero or both non-zero. Recall that $A$ is said
to be $1$-perfect if it has no ideals of codimension $1$ and an
ideal is called $1$-perfect if it is $1$-perfect as an algebra (see
\cite[Definition 5.3]{BavShk} for more details). We are now ready
to state our main results.
\begin{thm}
\label{thm:B'<eAf} Let $A$ be a $1$-perfect finite dimensional
associative algebra and $R$ be its radical. Let $A'\subseteq A$
be subalgebra and let $B'$ be a bar-minimal J-Lie of $A'$. Suppose
that $p\ne2,3$. Then there is a strict idempotent pair $(e,f)$ in
$A$ such that $B'\subseteq eAf$. 
\end{thm}

We also prove the following results
\begin{cor}
\label{cor:B'<B minimal and regular} Let $A$ be a $1$-perfect finite
dimensional associative algebra and $R$ be its radical. Let $A'\subseteq A$
be subalgebra and let $B'$ be a bar-minimal J-Lie of $A'$. Suppose
that $p\ne2,3$. Then 
\end{cor}

\begin{enumerate}
\item \emph{there is a bar-minimal J-Lie $B$ of $A$ such that $B'\subseteq B$.}
\item \emph{there is a regular inner ideal $B$ of $A$ such that $B'\subseteq B$.}
\end{enumerate}
\begin{cor}
\label{cor:e'Af'<eAf} Let $A$ be a $1$-perfect finite dimensional
associative algebra and $R$ be its radical. Let $A'\subseteq A$
be a perfect finite dimensional associative algebra and let $(e',f')$
be a strict orthogonal idempotent pair of $A'$. Then there is a strict
orthogonal idempotent pair $(e,f)$ in $A$ such that $e'A'f'\subseteq eAf$. 
\end{cor}

We believe that Jordan-Lie inner ideals of simple diagonal locally
finite Lie algebras can be classified using the results in this paper.
Recall that an algebra $A$ is called locally finite if every finite
set of elements is contained in a finite dimensional algebra (see
\cite{BahBavZal2004,bahStr1995,baranov1998} for more details. In
\cite{shlaka2020,Shlaka2021}, the researcher highlighted and proved
some results about locally finite associative algebras. Combining
the results of this paper together with the results proved in \cite{Shlaka2021},
it is believed the following conjecture is true.
\begin{conjecture}
Let $A$ be a simple locally finite associative algebra over $\mathbb{F}$
and let $B$ be a Jordan-Lie inner ideal of $L=[A,A]$. Then every
Jordan-Lie inner ideal $B$ of $L$ is regular.
\end{conjecture}

\section{Preliminaries}

   Throughout this paper, unless otherwise specified, the ground
field $\bbF$ is algebraically closed of arbitrary characteristic
$p$; $A$ is a finite dimensional associative algebra over $\bbF$;
$R=\rad A$ is the radical of $A$; $S$ is a Levi (i.e. maximal semisimple)
subalgebra of $A$, so $A=S\oplus R$; $L=A^{(k)}$ for some $k\ge0$;
$\rad L$ is the solvable radical of $L$ and $N=R\cap L$ is the
\emph{nil-radical} of $L$. If $V$ is a subspace of $A$, we denote
by $\bar{V}$ its image in $\bar{A}=A/R$. In particular, $\bar{L}=(L+R)/R\cong L/N$.
Since $R$ is a nilpotent ideal of $A$ the ideal $N=R\cap L$ of
$L$ is also nilpotent, so $N\subseteq\rad L$. It is easy to see
that $N=\rad L$ if $p=0$ and $k\ge1$, so $L/N$ is semisimple in
that case. Recall that $A$ is said to be \emph{separable} if for
every extension $\mathcal{F}$ of the field $\mathbb{F}$ one has
$A_{\mathcal{F}}=A\otimes\mathcal{F}$ is semisimple. In particular,
every separable algebra is semisimple \cite{Drozd}. The following
theorem is Wedderburn-Malcev Principal Theorem. For the proof see
\cite[Theorem 1]{BavMudShk2018}.
\begin{thm}[Wedderburn-Malcev]
 Let $A$ be a finite dimensional associative algebra and let $R$
be the radical of $A$. Suppose that $A/R$ is separable. Then the
following hold.

(1) There exists a semisimple subalgebra $S$ of $A$ such that $A=S\oplus R$
(Wedderburn). 

(2) If $Q$ is a semisimple subalgebra of $A$ then there exists $r\in R$
such that $Q\subseteq(1+r)S(1+r)^{-1}$ (Malcev).

(3) If $S_{1}$ and $S_{2}$ are two subalgebras of $A$ such that
$A=S_{i}\oplus R$ ($i=1,2$) then there exists $r\in R$ such that
$S_{1}=(1+r)S_{2}(1+r)^{-1}$ (Malcev). 
\end{thm}

The following theorem is proved in \cite[Theorem 6]{BavMudShk2018}.
\begin{thm}
\label{thm:BavMudShk Theorem6} Let $A$ be a finite dimensional algebra
and let $I$ be a left ideal of $A$. Suppose that $A/R$ is separable.
Then there exists a semisimple subalgebra $S$ of $A$ such that $A=S\oplus R$
and $I=I_{S}\oplus I_{R}$, where $I_{S}=I\cap S$ and $I_{R}=I\cap R$. 
\end{thm}

\begin{defn}
\label{def:(e,f) non-degen-1} (1) Let $A$ be a semisimple Artinian
ring and let $\{S_{i}\mid i\in I\}$ be the set of its simple components.
Let $e$ and $f$ be non-zero idempotents in $A$ and let $e_{i}$
(resp. $f_{i}$) be the projection of $e$ (resp. $f$) to $S_{i}$
for each $i\in I$. Then the pair $(e,f)$ is said to be\emph{ strict}
if for each $i\in I$, $e_{i}$ and $f_{i}$ are either both non-zero
or both zero.

(2) Let $A$ be an Artinian ring or a finite dimensional algebra and
let $R$ be its radical. Let $e$ and $f$ be non-zero idempotents
in $A$. We say that $(e,f)$ is \emph{strict} if $(\bar{e},\bar{f})$
is strict in $\bar{A}=A/R$. 
\end{defn}

\begin{lem}
\label{prop:eAf non 0 semi} Let $A$ be a semisimple Artinian ring
and let $(e,f)$ be a non-zero strict idempotent pair in $A$. Then
$eAf\ne0$. 
\end{lem}

\begin{thm}
\label{thm:Bavshk1.2} \cite[Theorem 1.2]{BavShk} Let $A$ be an
Artinian ring or a finite dimensional associative algebra over any
field and let $(e,f)$ and $(e',f')$ be idempotent pairs in $A$.
Suppose that $(e,f)$ is strict. Then the following hold.
\end{thm}

\begin{enumerate}
\item[(i)] \emph{ If $(e,f)\ne(0,0)$ then $eAf\ne0$.}
\item[(ii)] \emph{ Suppose that $eAf\subseteq e'Af'$. Then there exists a strict
idempotent pair $(e'',f'')$ in $A$ such that $e'e''=e''e'=e''$,
$f'f''=f''f'=f''$ and $e''Af''=eAf$.}
\end{enumerate}
Let $I$ be a left ideal of $A$ and let $Q$ be a left ideal of $\bar{A}$.
We say that $I$ is $Q$-\emph{minimal} (or simply \emph{bar-minimal})
if $\bar{I}=Q$ and for every left ideal $J$ of $A$ with $J\subseteq I$
and $\bar{J}=Q$ one has $J=I$.
\begin{prop}
\cite[Proposition 2]{BavMudShk2018} \label{Prop:BavMudShk Prop2}
Let $A$ be a left Artinian associative ring and let $I$ be a left
ideal of $A$. Suppose that $I$ is bar-minimal. Then there is an
idempotent $e\in I$ such that $I=Ae$. 
\end{prop}

\section{Inner Ideals}

Recall that an \emph{inner ideal} is a subspace $B$ of $L$ such
that $[B,[B,L]]\subseteq B$. 
\begin{lem}
\label{lem:Bav=000026Row 2.16} Let $L$ be a Lie algebra  and let
$B$ be an inner ideal of $L$. 

(i) If $M$ is a subalgebra of $L$, then $B\cap M$ is an inner ideal
of $M$.

(ii) If $P$ is an ideal of $L$, then $(B+P)/P$ is an inner ideal
of $L/P$. Moreover, for every inner ideal $Q$ of $L/P$ there is
a unique inner ideal $C$ of $L$ containing $P$ such that $Q=C/P$.
\end{lem}

\begin{lem}
\label{lem:Bar=000026Row 4.1-1} \cite{BavShk} Let $L=A^{(k)}$ for
some $k\ge0$ and let $B$ be a subspace of $L$ such that $B^{2}=0$.
Then the following hold. 
\end{lem}

\begin{enumerate}
\item \emph{If $p\ne2$ then $B$ is an inner ideal of $L$ if and only
if $bLb\subseteq B$ for all $b\in B$. }
\item \emph{$BAB\subseteq L\cap A^{(1)}$.}
\end{enumerate}
Let $B$ be an inner ideal of $L$. Then $[B,[B,L]]\subseteq B$.
It is well known that $[B,[B,L]]$ is an inner ideal of $L$ (see
for example \cite[Lemma 1.1]{Benkart1977}). Put $B_{0}=B$ and consider
the following inner ideals of $L$: 
\begin{equation}
B_{n}=[B_{n-1},[B_{n-1},L]]\subseteq B_{n-1}\quad\text{for}\quad n\geq1.\label{eq:Bn}
\end{equation}
Then $B=B_{0}\supseteq B_{1}\supseteq B_{2}\supseteq\ldots$. As $L$
is finite dimensional, this series terminates, so there is integer
$n$ such that $B_{n}=B_{n+1}$. Recall that a Lie algebra $L$ is
\emph{perfect} if $[L,L]=L$. 
\begin{defn}
\cite{BavShk} \label{def:Perfect} Let $L$ be a Lie algebra and
let $B$ be an inner ideal of $L$. 

1. We say that $B$ is \emph{$L$-perfect} if $B=[B,[B,L]]$. 

2. If $L$ is finite dimensional, then $B_{n}$ above is called \emph{the}
\emph{core} \emph{of} $B$, denoted by $\core_{L}(B)$.
\end{defn}

\begin{lem}
\label{lem:cor is perfect} \cite{BavShk} Let $L$ be a finite dimensional
Lie algebra and let $B$ be an inner ideal of $L$. Then
\end{lem}

\begin{enumerate}
\item If $B$ is $L$-perfect, then $B$ is an inner ideal of $L^{(k)}$
for all $k\ge0$.
\item $\core_{L}(B)$ is $L$-perfect;
\item $B$ is $L$-perfect if and only if $B=\core_{L}(B)$;
\item $\core_{L}(B)$ is an inner ideal of $L^{(k)}$ for all $k\ge0$.
\end{enumerate}
\begin{defn}
\label{def:minimal} Let $L=A^{(k)}$ and let $B$ be an inner ideal
of $L$. Then 

1) $B$ is called \emph{Jordan-Lie }(or simply, \emph{J-Lie})\emph{
}if $B^{2}=0$.

2) $B$ is called \emph{regular} if $B^{2}=0$ and $BAB\subseteq B$. 

3) Let $X$ be an inner ideal of $\bar{L}=L/\rad L$. We say that
$B$ is \emph{$X$-minimal} (or simply, \emph{bar-minimal})\emph{
}if $\bar{B}=X$ and for every inner ideal $B'\subseteq B$ of $L$
with $\bar{B}'=X$ we have $B'=B$.
\end{defn}

It follows from Benkart's result \cite[Theorem 5.1]{Benkart1976}
that if $A$ is a simple associative algebra and $p\ne2,3$, then
every proper inner ideal $B$ of $[A,A]/(Z(A)\cap[A,A])$ has the
following property $B^{2}=0$, so it is Jordan-Lie.
\begin{lem}
\label{lem:Jordan-Lie -1} Let $L=A^{(k)}$ for some $k\ge0$ and
let $B$ be a subspace of $L$. Then $B$ is a Jordan-Lie inner ideal
of $L$ if and only if $B^{2}=0$ and 
\[
\{b,x,b'\}:=bxb'+b'xb.\in B
\]
 for all $b,b'\in B$ and $x\in L$. 
\end{lem}

\begin{lem}
\label{lem:B is L-perfect if A is semi} \cite[Lemma 5.10]{BavShk}
Suppose $A$ is semisimple, $k\ge0$ and $p\neq2,3$. Then every J-Lie
of $L=A^{(k)}$ is $L$-perfect.
\end{lem}

\begin{rem}
Let $k\ge0$. If $S$ is a Levi subalgebra of $A$, then $A=S\oplus R$,
so $A^{(k)}=S^{(k)}\oplus N$, where $N=R\cap A^{(k)}$. Moreover,
$\bar{A}^{(k)}=A^{(k)}/N=A^{(k)}/R\cap A^{(k)}$ is the image of $A^{(k)}$
in $\text{\ensuremath{\bar{A}=A/R}}$. 
\end{rem}

\begin{lem}
\label{B=00003Dcor(B)} \cite[Lemma 5.16]{BavShk} Let $B$ be a J-Lie
inner ideal of $L=A^{(k)}$ ($k\ge0$). If $p\neq2,3$, then 

(i) $\bar{B}=\overline{\core_{L}(B)}$.

(ii) If $\core_{L}(B)=0$, then $B\subset N$.
\end{lem}

\begin{prop}
\label{prop:Bav=000026Row 4.9-1}\cite[Proposition 4.9]{BavRow2013}
Let $B$ be a subspace of $L=A^{(k)}$ ($k\ge0$). Then $B$ is a
regular inner ideal of $L$ if and only if there exist left and right
ideals $\ccL$ and $\ccR$ of $A$, respectively, such that $\ccL\ccR=0$
and 
\[
\ccR\ccL\subseteq B\subseteq\ccR\cap\ccL.
\]
In particular, if $A$ is Von Neumann regular then every regular inner
ideal of $L$ is of the form $B=\ccR\ccL=\ccR\cap\ccL$.
\end{prop}

\begin{lem}
\label{lem:Bar=000026Row 4.1} \cite{BavShk} Let $L=A^{(k)}$ for
some $k\ge0$ . Then the following hold. 
\end{lem}

\begin{enumerate}
\item \emph{If $eAf\subseteq A^{(k)}$ ($k\ge0$), then $eAf$ is a regular
inner ideal of $A^{(k)}$. }
\item \emph{$eAf$ is a regular inner ideal of $A^{(-)}$ and $A^{(1)}$.}
\item \emph{Suppose $A$ is semisimple, $p\neq2,3$ and $k\ge0$. Then every
}J-Lie\emph{ inner ideal of $A^{(k)}$ is regular. }
\item \emph{Let $k\ge0$ and let $B$ be a J-Lie inner ideal of $L=A^{(k)}$.
Suppose that $B$ is bar-minimal and $p\neq2,3$. Then the following
hold. }
\begin{enumerate}
\item \emph{$B=\core_{L}B$. }
\item \emph{$B$ is $L$-perfect.}
\item \emph{$B$ is a J-Lie inner ideal of $L^{(m)}=A^{(k+m)}$ for all
$m\ge0$.}
\item \emph{If $B$ is regular then $B=eAf$ for some orthogonal idempotent
pair $(e,f)$ in $A$. }
\item \emph{(ii) Suppose $k=0,1$. Then $B$ is regular if and only if $B=eAf$
for some orthogonal idempotent pair $(e,f)$ in $A$.}
\end{enumerate}
\end{enumerate}
\begin{lem}
\label{lem:B=00003DB1+B2 if L=00003DL1+L2} \cite{BavShk} Let $L$
be a perfect Lie algebra and let $B$ be an $L$-perfect inner ideal
of $L$. Suppose that $L=\bigoplus_{i\in I}L_{i}$, where each $L_{i}$
is an ideal of $L$. Then $B=\bigoplus_{i\in I}B_{i}$, where $B_{i}=B\cap L_{i}$.
Moreover, if $L=A^{(k)}$ ($k\ge0$) and $B$ is bar-minimal then
$B_{i}$ is a $\bar{B}_{i}$-minimal inner ideal of $L_{i}$, for
all $i\in I$.
\end{lem}

\begin{thm}
\label{thm:BavShkTheorem1.1} \cite[Theorem 1.1]{BavShk} Let $A$
be a finite dimensional associative algebra (over an algebraically
closed field of arbitrary characteristic $p$) and let $B$ be a J-Lie
of $L=A^{(k)}$ ($k\ge0$). Suppose $p\neq2,3$. Then $B$ is bar-minimal
if and only if $B=eAf$ where $(e,f)$ is a strict orthogonal idempotent
pair in $A$.
\end{thm}

\section{Inner Ideals and Splitness}
\begin{defn}
An associative algebra is said to be $1$-\emph{perfect} if it has
no ideals of codimension 1. An ideal is said to be $1$-\emph{perfect}
if it is $1$-perfect as an algebra.
\end{defn}

By using 2nd and 3rd Isomorphism Theorems we get the following properties
of $1$-perfect ideals.
\begin{lem}
\label{1p-1} \cite{BavShk} (i) The sum of $1$-perfect ideals is
$1$-perfect.

(ii) If $P$ is a $1$-perfect ideal of $A$ and $Q$ is a $1$-perfect
ideal of $A/P$ then the full preimage of $Q$ in $A$ is a $1$-perfect
ideal of $A$.
\end{lem}

Lemma \ref{1p-1}(i) implies that every algebra has the largest $1$-perfect
ideal. 
\begin{defn}
\label{def:1-perfect radical-1} The largest $1$-perfect ideal $\ccP_{A}$
of $A$ is called the\emph{ }$1$\emph{-perfect radical }of $A$.
\end{defn}

\begin{lem}
\label{lem:A^(k)=00003Dp(A)} Let $p\ne2$. Then $A^{(\infty)}=[\ccP_{A},\mathcal{P}_{A}]$.
\end{lem}

\begin{thm}
\label{BaranovLiePerfect} \cite[Theorem 1.1(1)]{Baranov2018Arx}
If $A$ is $1$-perfect and $p\ne2$, then $[A,A]$ is a perfect Lie
algebra.
\end{thm}

\begin{defn}
\label{def:large } \cite[Definition 6.10]{BavShk} Let $G$ be a
subalgebra of $A$. We say that $G$ is \emph{large in $A$ }if $G+R=A$
(equivalently, there is a Levi subalgebra $S$ of $A$ such that $S\subseteq G$;
or equivalently, $G/\rad G$ is isomorphic to $A/R$).
\end{defn}

\begin{rem}
\label{rem:rad_large-1} Let $G$ be a large subalgebra of $A$ and
let $B$ be a subspace of $\ccP_{G}$. Then $\rad(G)=G\cap R$ and
$\rad(\ccP_{G})=\ccP_{G}\cap\rad(G)=\ccP_{G}\cap R$, so the image
$\bar{B}$ of $B$ in $A/R$ is isomorphic to the images of $B$ in
$G/\rad(G)$ and $\ccP_{G}/\rad(\ccP_{G})$, respectively. Thus, we
can use the same notation $\bar{B}$ for the images of $B$ in all
these quotient spaces.
\end{rem}

\begin{defn}
\label{def:splitness} Let $B$ be a subspace of $A$. Suppose that
there is a Levi subalgebra $S$ of $A$ such that $B=B_{S}\oplus B_{R}$,
where $B_{S}=B\cap S$ and $B_{R}=B\cap R$. Then we say that $B$
\emph{splits} in $A$ and $S$ is a $B$\emph{-splitting} Levi subalgebra\emph{
}of $A$.
\end{defn}

\begin{thm}
\label{cor:cor(B) splits} \cite{BavShk}Let $G$ be a large subalgebra
of $A$ and let $B$ be a J-Lie of $L=A^{(k)}$ ($k\ge0$). Suppose
that $p\ne2,3$. The following hold.
\end{thm}

\begin{enumerate}
\item If $\core_{L}(B)$ splits in $A$, then $B$ splits in $A$.
\item If $B'=B\cap G^{(k)}$ has the property $\bar{B}'=\bar{B}$. Then
$C=\core_{G^{(k)}}(B')$ is a J-Lie of $\ccP_{1}(G)^{(1)}$ such that
$C\subseteq B$ and $\bar{C}=\bar{B}$. 
\end{enumerate}
\begin{prop}
\label{prop:C sub B splits in A} \cite{BavShk} Let $B$ be a subspace
of $A$ and let $G$ be a large subalgebra of $A$. Suppose that $C_{A}\subseteq B$
and $C_{\mathcal{P}_{G}}\subseteq B$ are subspaces of $A$ and $\mathcal{P}_{G}$,
respectively, such that $\bar{C}_{A}=\bar{B}$ and \emph{$\bar{C}_{\mathcal{P}_{G}}=\bar{B}$.
The following hold.}
\end{prop}

\begin{enumerate}
\item \emph{If $C_{A}$ splits in $A$, then $B$ splits in $A$.}
\item \emph{If $C_{\mathcal{P}_{G}}$ splits in $\ccP_{G}$. Then $B$ splits
in $A$.}
\end{enumerate}
Since $A$ is large in $A$, we get the following corollary. 

\begin{defn}
\cite{BavShk} Let $Q$ be a Lie algebra. We say that $Q$ is a \emph{quasi
(semi)simple} if $Q$ is perfect and $Q/Z(Q)$ is (semi)simple. 
\end{defn}

Herstein \cite{Herstein1955} and \cite{Herstein1961} proved that
if $A$ is a simple ring of characteristic different from $2$, then
$A^{(1)}=[A,A]$ is a quasi simple Lie ring. As a special case, we
note the following well-known fact. Recall that $\ccM_{n}$ is the
algebra of $n\times n$ matrices over $\bbF$ and by $\sl_{n}$ the
Lie subalgebra of $\ccM_{n}$ consisting of zero trace matrices, so
$\sl_{n}=\ccM_{n}^{(1)}$.
\begin{example}
\label{lem:sl_n-1} If $p\ne2$, then $\sl_{n}=[\mathcal{M}_{n},\mathcal{M}_{n}]$
(for all $n\ge2$) is quasi simple Lie subalgebra of the algebra $\ccM_{n}$. 
\end{example}

Note that the case of $p=2$ is exceptional indeed as the algebra
$\sl_{2}$ is solvable in characteristic $2$. 
\begin{prop}
\label{prop:=00005BA,A=00005D quasi Levi-1} \cite{BavShk} Suppose
$A$ is semisimple and $p\neq2$. Then $[A,A]$ is quasi semisimple.
In particular, $A^{(\infty)}=A^{(1)}$. 
\end{prop}

\begin{defn}
\label{def:quasi Levi-1} Let $M$ be a finite dimensional Lie algebra
and let $Q$ be a quasi semisimple subalgebra of $M$. We say that
$Q$ is a \emph{quasi Levi subalgebra} of $M$ if there is a solvable
ideal $P$ of $M$ such that $M=Q\oplus P$. In that case we say that
$M=Q\oplus P$ is a \emph{quasi Levi decomposition} of $M$. 
\end{defn}

\begin{defn}
\label{def:splitness-1} Let $L$ be a finite dimensional Lie algebra
and let $B$ be a subspace of $L$. Suppose that there is a quasi
Levi decomposition $L=Q\oplus N$ of $L$ such that $B=B_{Q}\oplus B_{N}$,
where $B_{Q}=B\cap Q$ and $B_{N}=B\cap N$. Then we say that $B$
\emph{splits} \emph{in} $L$ and $Q$ is a $B$\emph{-splitting} \emph{quasi
Levi subalgebra }of $L$. 
\end{defn}

\begin{lem}
\label{lem:If split in A, then split in L} Let $L=A^{(k)}$ ($k\ge1$)
and let $B$ be a subspace of $L$. Suppose $p\ne2$. If $B$ splits
in $A$, then $B$ splits in $L$. 
\end{lem}

\begin{lem}
\label{lem:eAf splits} Let $B$ be an inner ideal of $L=A^{(k)}$
($k\ge0$). Suppose $B=eAf$ for some orthogonal idempotents $e$
and $f$ of $A$. Then (i) $B$ splits in $A$ and (ii) if $k\ge1$
then $B$ splits in $L$. 
\end{lem}

\section{The Main Results}

Recall that $A$ is a finite dimensional associative algebra over
$\mathbb{F}$. If $\mathcal{I}_{A}$ is an ideal of $A$, then $\mathcal{I}_{A}^{(-)}$
is a Lie ideal (and subalgebra as well) of the Lie algebra $A^{(-)}$
under the Lie multiplication define by $[a,b]=ab-ba$ for all $a,b\in A$. 
\begin{thm}
\label{thm:B=00003DeAf bar-minimal of P} Let $\mathcal{I}_{A}$ be
an ideal of $A$ and let $B$ be a bar-minimal Jordan-Lie inner ideal
of $\mathcal{I}_{A}^{(1)}=[\mathcal{I}_{A}^{(-)},\mathcal{I}_{A}^{(-)}]$.
Then there is a strict orthogonal idempotent pair $(e,f)$ in $A$
such that $B=eAf$.
\end{thm}

\begin{proof}
Since $B$ is bar-minimal, $B=e\mathcal{I}_{A}f$ for some strict
orthogonal idempotent pair $(e,f)$ of $\mathcal{I}_{A}$. As $\mathcal{I}_{A}$
is an ideal of $A$, 
\[
BAB=e\mathcal{I}_{A}fAe\mathcal{I}_{A}f\subseteq e\mathcal{I}_{A}f=B.
\]
Hence, by Lemma \ref{lem:Bar=000026Row 4.1-1}, $B\subseteq[A,A]$,
so it is an inner ideal of $[A,A]$. As $e,f\in\mathcal{I}_{A}$,
we have

\[
e\mathcal{I}_{A}f\subseteq eAf=eeAf\subseteq e\mathcal{I}_{A}Af\subseteq e\mathcal{I}_{A}f.
\]
Therefore, $B=e\mathcal{I}_{A}f=eAf$, as required. 
\end{proof}
\begin{prop}
\label{prop:B'<eAf  if R^2=00003D0} Let $A$ be a $1$-perfect finite
dimensional associative algebra and $R$ be its radical. Let $A'\subseteq A$
be a perfect finite dimensional associative algebra and let $B'$
be a bar-minimal J-Lie inner ideal of $A'$. Suppose that $p\ne2,3$
and $R^{2}=0$. Then there is a strict idempotent pair $(e,f)$ in
$A$ such that $B'\subseteq eAf$.
\end{prop}

\begin{proof}
Since $B'$ is bar-minimal, by Theorem \ref{thm:BavShkTheorem1.1},
there is a strict orthogonal idempotent pair $(e',f')$ in $A'$ such
that $B'=e'A'f'$. Let $\ccL'=A'f'$ and $\ccR'=e'A'$. Then $\ccL'$
and $\ccR'$ be left and right ideals of $A'$, respectively, with
$B'=\ccR'\ccL'$ and $\ccL'\ccR'=0$. Fix a Levi subalgebra $S$ of
$A$. Then $A=S\oplus R$. Since $f'\in A'\subseteq A=S\oplus R$
, we have $f'=s+r$ for some $s\in S$ and $r\in R$. As $f'=f'f'$,
we have $s+r=f'=f'f'=ss+sr+rs+r^{2}$, so $s=ss$ and $r=rs+sr$.
Note that $srs=0$. Indeed, $f'f'=(s+sr+rs)(s+sr+rs)=f'+2srs$, so
$srs=0$. Put $f=s+sr$. Then 
\[
f^{2}=(s+sr)(s+sr)=s+sr+srs=s+sr=f,
\]
so $f$ is an idempotent with 
\[
f'f=(s+sr+rs)(s+sr)=s+sr+rs=f'
\]
 and 
\[
ff'=(s+sr)(s+sr+rs)=s+sr+2srs=s+sr=f.
\]
Similarly, there is idempotent $e$ in $A$ with $ee'=e'$ and $e'e=e$.
Therefore, $B'=e'A'f'\subseteq e'Af'=ee'Af'f\subseteq eAf$. Moreover,
we have $fe=(ff')(e'e)=f(f'e')e=0$, so by Lemma \ref{lem:Bar=000026Row 4.1},
one can assume that $(e,f)$ is an orthogonal idempotent pair. It
remains to show that $(e,f)$ is strict. Since $(e',f')$ strict and
$(e',f')\overset{\mathcal{LR}}{\le}(e,f)$ (because $e'Af'\subseteq eAf$,
by Theorem \ref{thm:Bavshk1.2}, there is a strict idempotent pair
$(e'',f'')$ in $A$ such that $e''Af''=eAf$, as required.
\end{proof}
Now, we are ready to prove our main results. 
\begin{proof}[Proof of Theorem \ref{thm:B'<eAf}]
 Recall that $p\ne2,3$, $A$ is $1$-perfect, $A'\subseteq A$ is
a subalgebra and $B'$ is bar-minimal J-Lie of $A'$. We need to prove
that there is a strict idempotent pair $(e,f)$ in $A$ such that
$B'\subseteq eAf$.

Recall that $B'$ is bar-minimal. Since $R=\rad A$ is nilpotent,
there is an integer $m$ such that $R^{m-1}\neq0$ and $R^{m}=0$.
The proof is by induction on $m$. If $m=2$ , then by Proposition
\ref{prop:B'<eAf  if R^2=00003D0}, $B'\subseteq eAf$ for some idempotent
pair $(e,f)$ in $A$. 

Suppose that $m>2$. Put $T=R^{2}\neq0$ and consider $\tilde{A}=A/T$.
Let $\tilde{A}'$, $\tilde{B}'$ and $\tilde{R}$ be the images of
$\tilde{A}'$, $B'$ and $R$ in $\tilde{A}$. Since $\tilde{R}^{2}=R^{2}=0$
in $\tilde{A}$, by Proposition \ref{prop:B'<eAf  if R^2=00003D0},
$\tilde{B}'\subseteq\tilde{e}\tilde{A}\tilde{f}$ for some strict
orthogonal idempotent pair $(\tilde{e},\tilde{f})$ in $\tilde{A}.$
By Theorem, there is a Levi subalgebra $S$ of $\tilde{A}$ such that
$\tilde{e},\tilde{f}\in S$. Thus, $\tilde{B}'\subseteq\tilde{e}S\tilde{f}\oplus\tilde{e}S\tilde{f}$.
Let $P$ be the full preimage of $\tilde{e}S\tilde{f}$ in $B'$.
Then $\tilde{P}=\tilde{e}S\tilde{f}\subseteq S$, so $P$ is a subspace
of $B'$ with $\bar{P}=\bar{B'}$. Let $G'$ and $G$ be the full
preimages of $S$ in $A'$ and $A$, respectively. Then $G'$ and
$G$ are large subalgebras of $A'$ and $A$, respectively, with $G'\subseteq G$
and $P\subseteq G'\cap B'$. Let $S'$ be a Levi subalgebra of $G'$
such that $S'\subseteq S$. Put $P_{1}=[P,[P,S'{}^{(1)}]]$ and $B_{1}=B'\cap G'^{(1)}$.
Then $P_{1}\subseteq[B',[B',A'^{(1)}]]\subseteq B'$ and $P_{1}\subseteq[G',[G',G']]\subseteq G'^{(1)}$,
so $P_{1}\subseteq B'\cap G'^{(1)}=B_{1}$. Since 
\[
\bar{P}_{1}=[\bar{P},[\bar{P},\bar{S}'^{(1)}]]=[\bar{B}',[\bar{B}',\bar{A}'^{(1)}]]=\bar{B}',
\]
we get that $\bar{B'}=\bar{P}_{1}\subseteq\bar{B}_{1}\subseteq\bar{B'}$,
so $\bar{B}_{1}=\bar{B}'$. As $G'^{(1)}$ is a Lie subalgebra of
$A'^{(1)}$, $B_{1}=B'\cap G'^{(1)}$ is a J-Lie of $G'^{(1)}\subseteq G^{(1)}$.
Put $B_{2}=\core_{G'^{(1)}}(B_{1})$. Then by Proposition \ref{cor:cor(B) splits},
$B_{2}$ is a J-Lie of $\ccP_{G'}^{(1)}$ such that $B_{2}\subseteq B'$
and $\bar{B}_{2}=\bar{B}'$. Let $B_{3}\subseteq B_{2}$ be any $\bar{B}_{2}$-minimal
inner ideal of $\mathcal{P}_{G'}^{(1)}$. Since $\ccP_{G'}\subseteq\mathcal{P}_{G}$
is $1$-perfect and $\rad(\ccP_{G})^{m-1}\subseteq T^{m-1}=R^{2(m-1)}=0$,
by the inductive hypothesis, there is a strict idempotent pair $(e,f)$
in $\mathcal{P}_{G}$ such that $B_{3}\subseteq e\mathcal{P}_{G}f$.
By ?, $eAf=e\mathcal{P}_{G}f$, so $B_{3}\subseteq eAf$ for some
orthogonal idempotent pair $(e,f)$ in $A$. Thus, $B_{3}\subseteq B'\cap eAf$.
Since $\bar{B}'=\bar{B}_{3}\subseteq\overline{B'\cap eAf}\subseteq\bar{B}'$.
Hence, $\bar{B}'=\overline{B'\cap eAf}=\bar{B}'\cap\bar{e}\bar{A}\bar{f}$.
Note that $B'\cap eAf$ is a J-Lie of $A'$ with $B'\cap eAf\subseteq B'$.
Since $\bar{B}'=\overline{B'\cap eAf}$ and $\bar{B}'$ is bar-minimal,
$B'\cap eAf=B'$. Therefore, $B'=B'\cap eAf\subseteq eAf$, as required.
\end{proof}
\begin{proof}[Proof of Corollary \ref{cor:B'<B minimal and regular}]
 Recall that $p\ne2,3$, $A$ is $1$-perfect, $A'\subseteq A$ is
subalgebra and $B'$ is a bar-minimal J-Lie of $A'$. We need to prove
that 

1. There is a bar-minimal J-Lie $B$ of $A$ such that $B'\subseteq B$.

2. There is a regular inner ideal $B$ of $A$ such that $B'\subseteq B$.

By Theorem \ref{thm:B'<eAf}, $B'\subseteq eAf$ for some strict orthogonal
idempotent pair $(e,f)$ in $A$. 

For the proof of (1), by Lemma \ref{thm:BavShkTheorem1.1} we have
$eAf$ is a bar-minimal J-Lie of $A$. Therefore, $B=eAf$ is the
required bar-minimal J-Lie of $A$.

For the proof of (2), by Lemma \ref{lem:Bar=000026Row 4.1}, we have
$eAf$ is a regular inner ideal of $A$. Hence $B=eAf$ is the required
regular inner ideal of $A$. 
\end{proof}
\begin{proof}[Proof of Corollary \ref{cor:e'Af'<eAf}]
 Recall that $p\ne2,3$, $A$ is $1$-perfect and $A'\subseteq A$
is subalgebra. We need to prove that for every strict orthogonal idempotent
pair $(e',f')$ in $A'$, there is a strict orthogonal idempotent
pair $(e,f)$ in $A$ such that $e'A'f'\subseteq eAf$. By Lemma \ref{thm:BavShkTheorem1.1},
$e'A'f'$ is a bar-minimal J-Lie of $A'.$Since $A$ is $1$-perfect,
by Theorem \ref{thm:B'<eAf}, there is a strict orthogonal idempotent
pair $(e,f)$ in $A$ such that $e'A'f'\subseteq eAf$, as required.
\end{proof}

\subsubsection*{}

\end{document}